\documentclass{article}
\usepackage{multicol,graphicx,color}
\usepackage{pslatex}
\usepackage{authblk}
\usepackage{amsthm}
\usepackage{amsmath}
\usepackage{amssymb}
\usepackage{latexsym}
\usepackage{lscape}
\usepackage{epsfig}
\usepackage{pstricks}
\usepackage{amsfonts}
\usepackage{mathrsfs}
\usepackage{mathrsfs}
\usepackage[
  hmarginratio={1:1},     
  vmarginratio={1:1},     
  textwidth=15cm,        
  textheight=21cm,
  heightrounded,          
]{geometry}

\usepackage{graphicx,color}
\usepackage[colorlinks]{hyperref}
\hypersetup{linkcolor=blue,citecolor=blue,filecolor=black,urlcolor=blue}

\theoremstyle{plain}
\newtheorem{theorem}{Theorem}

\newtheorem{lemma}[theorem]{Lemma}
\newtheorem{proposition}[theorem]{Proposition}
\theoremstyle{definition}
\newtheorem{definition}[theorem]{Definition}
\newtheorem{remark}[theorem]{Remark}

\numberwithin{equation}{section}
\numberwithin{theorem}{section}

\newcommand{\cG}{\mathcal{G}}
\newcommand{\R}{\mathbb{R}}
\newcommand{\eps}{\varepsilon}
\newcommand{\dist}{{\rm dist}}

\newcommand{\edge}{{\rm e}}
\newcommand{\pa}{\partial}

\newcommand{\M}{S_{\mu}}

\author{Jack Borthwick\footnote{jack.borthwick@math.cnrs.fr}}  \affil{Universit\'e de Franche-Comt\'e, Laboratoire de Math\'ematiques de Besan\c{c}on, UMR CNRS 6623, 16 route de Gray, 25000 Besan\c{c}on, France}

\author{Xiaojun Chang \footnote{changxj100@nenu.edu.cn}} \affil{School of Mathematics and Statistics \& Center for Mathematics and Interdisciplinary Sciences,
 Northeast Normal University, Changchun 130024, Jilin,
PR China}

\author{Louis Jeanjean\footnote{louis.jeanjean@univ-fcomte.fr}}  \affil{Universit\'e de Franche-Comt\'e, Laboratoire de Math\'ematiques de Besan\c{c}on, UMR CNRS 6623, 16 route de Gray, 25000 Besan\c{c}on, France}

\author{Nicola Soave\footnote{nicola.soave@unito.it}}  \affil{Dipartimento di Matematica ``Giuseppe Peano", Universit\`a degli Studi di Torino, 10123 Torino, Italy}

\title{Normalized solutions of $L^2$-supercritical NLS equations on noncompact metric graphs with localized nonlinearities}
\date{}

\begin{document}

\maketitle

\begin{abstract}
\noindent In this paper we are concerned with the existence of normalized solutions for nonlinear Schr\"odinger equations on noncompact metric graphs with localized nonlinearities. In a $L^2$-supercritical regime, we obtain the existence of solutions for any prescribed mass. This result is obtained through  an approach which could prove successful to treat more general equations on noncompact graphs.
\end{abstract}

\medskip

{\small \noindent \text{Keywords:} Nonlinear Schr\"odinger equations; $L^2$-supercritical;
Noncompact metric graphs; Localized nonlinearities; Variational methods.}

\medskip

{\small \noindent \text{Acknowledgements:}
X. J. Chang is partially supported by NSFC (11971095).
N. Soave is partially supported by the INDAM-GNAMPA group.
\\ J. Borthwick gratefully acknowledges that part of this work was supported by the French Investissements d'Avenir program, project ISITE-BFC (contract ANR- 15-IDEX-0003).}

\section{Introduction and main results}\label{intro}

This paper is devoted to the existence of bound states of prescribed mass for the $L^2$\emph{-supercritical} nonlinear Schr\"odinger equation (NLS) on $\cG$
\begin{equation}\label{stat nls}
-u''+\lambda u=\kappa(x)|u|^{p-2}u,
\end{equation}
coupled with Kirchhoff conditions at the vertices, see \eqref{1.2} below. Here $\lambda$ appears as a Lagrange multiplier, $p>6$, $\cG$ is a noncompact connected metric graph, and $\kappa$ is the characteristic function of the compact core $\mathcal{K}$ of $\cG$. We recall, see \cite{AST-JFA2016}, that if $\cG$ is a metric graph with a finite number of vertices  its compact core $\mathcal{K}$ is defined as the metric subgraph of $\cG$  consisting of all the bounded edges of $\cG$.

A solution with prescribed mass  is often referred to as a \emph{normalized solution}. For Equation (\ref{stat nls}) they correspond to
critical points of the NLS energy functional $E(\cdot\,,\mathcal{G}):H^1(\cG) \to \R$ defined by
\begin{eqnarray}\label{1.1}
E(u,\mathcal{G})=\frac{1}{2}\int_{\mathcal{G}}|u'|^2\, dx-\frac{1}{p}\int_{\mathcal{K}}|u|^p\, dx
\end{eqnarray}
under a mass constraint
\begin{equation}\label{mass const}
\int_{\mathcal{G}}|u|^2\,dx=\mu>0.
\end{equation}

The study of NLS on metric graphs has attracted much attention over the last few decades. Among the physical motivations to consider Schr\"odinger equations on a metric graph, the most prominent one comes from the study of quantum graphs, see for example \cite{AST11,BK,KNP,No} and the references therein. In particular, the investigation of NLS with localized nonlinearities appeared in optical fibers and Bose-Einstein condensates in non-regular traps (see \cite{GSD-PRA-2011,No}).

 Recently, much effort has been devoted to the existence of normalized solutions of NLS on metric graphs, in the $L^2$\emph{-subcritical} (i.e., $p\in(2,6)$) or $L^2$\emph{-critical regimes} (i.e., $p=6$). In these two regimes, the energy functional $E(\cdot, \cG)$ is bounded from below on the mass constraint and one may develop  minimization methods. A relevant notion is indeed the one of ground states, namely of solutions which minimize the energy functional on the constraint. For the existence of ground states solutions, see \cite{ABD-JFA2022,ACFN-2014, AST-CVPDE2015,AST-JFA2016,AST-CMP2017,NP,PiSo} for noncompact $\cG$, and \cite{CDS, D-JDE2018} for the compact case; some studies are also conducted on the existence of local minimizers, see \cite{PSV}. Regarding problems with a localized nonlinearity as in Equation (\ref{stat nls}), existence or non-existence of ground states solutions was discussed in \cite{ST-NA2016,T-JMAA2016} for the $L^2$-subcritical case, and in \cite{DT-CVPDE-2019,DT-OTAA-2020} for the $L^2$-critical case. Moreover, in the $L^2$-subcritical regime, genus theory was applied in \cite{ST-JDE2016} to obtain the existence of multiple bound states solutions at negative energy levels.

In the $L^2$-supercritical regime on general metric graphs, i.e., when $p>6$, the energy functional $E(\cdot, \cG)$ is always unbounded from below.
Moreover, due to the fact that graphs are not scale invariant, the techniques based on scalings, usually employed in the Euclidean setting and related to the validity of a Pohozaev identity (see \cite{J} or \cite{BaSo, IkNo, S-JDE-2020, S-JFA-2020}), do not work. These two features make the search for normalized solutions in the 
 $L^2$-supercritical regime delicate. Recently, in \cite{CJS-2022}, this issue was considered on compact metric graphs. The existence of solutions was proved for small values of $\mu >0$, see \cite[Theorem 1.1]{CJS-2022} for a precise statement. In the present paper we pursue this study by allowing the graph to be noncompact. Although in our proofs we make use of the fact that the nonlinearity in Equation (\ref{stat nls}) only acts on a compact part of the graph, we believe that our approach could prove fruitful to study more general situations.

\subsection*{Basic notations and main result}

Throughout this paper, we consider a connected noncompact metric graph $\cG = (\mathcal{E}, \mathcal{V})$, where $\mathcal{E}$ is the set of edges and $\mathcal{V}$ is the set of vertices. We assume that $\cG$ has a finite number of edges and the compact part $\mathcal{K}$ is non-empty.
Without loss of generality, we may assume that $\{{\rm e}_1, \cdots, {\rm e}_{m_1}\}$ are the bounded edges, $\{l_1, \cdots, l_{m_2}\}$ are the half-lines, and $\{{\rm{v}}_1, \cdots, {\rm{v}}_{m_3}\}$ are the finite vertices of $\mathcal{G}$. Here $m_1, m_2, m_3$ are positive integers.
Moreover, any bounded edge ${\rm e}$ is identified with a closed bounded interval $I_{\rm e}$, typically $[0,\ell_{\rm e}]$ (where $\ell_{\rm e}$ is the length of ${\rm e}$), while each unbounded edge is identified with a closed half-line $I_{\rm e}=[0,+\infty)$. For details, one can refer to \cite{AST-CVPDE2015,AST-JFA2016,BK}. A function on metric graph $u: \cG \to \R$ is identified with a vector of functions $\{u_{\mathrm{e}}\}$, where each $u_{\mathrm{e}}$ is defined on the corresponding interval $I_{\rm e}$ such that $u|_{\mathrm{e}}=u_{\mathrm{e}}$. Endowing each edge with Lebesgue measure, one can define the space $L^p(\cG)$ in a natural way, with norm
\[
\|u\|_{L^p(\cG)}^p = \sum_{\mathrm{e}} \|u_{\mathrm{e}}\|_{L^p(\mathrm{e})}^p.
\]
The Sobolev space $H^1(\cG)$ consists of the set of continuous functions $u: \cG \to \R$ such that $u_{\mathrm{e}} \in H^1([0, \ell_{\mathrm{e}}])$ for every edge $\mathrm{e}$ (note that $\ell_{\mathrm{e}}$ can be $+\infty$); the norm in $H^1(\cG)$ is defined as
\[
\|u\|_{H^1(\cG)}^2 = \sum_{\mathrm{e}} \|u_{\mathrm{e}}'\|_{L^2(\mathrm{e})}^2 + \|u_{\mathrm{e}}\|_{L^2(\mathrm{e})}^2.
\]

We shall study the existence of critical points of the functional $E(u,\mathcal{G})$ constrained on the $L^2$-sphere
\[
H_{\mu}^1(\mathcal{G}):= \left\{u\in H^1(\mathcal{G}): \int_{\mathcal{G}}|u|^2\,dx=\mu\right\}.
\]
If $u\in H^1_{\mu}(\mathcal{G})$ is such a critical point, then there exists a Lagrange multiplier $\lambda\in \mathbb{R}$ such that $u$ satisfies the following problem:
\begin{equation}\label{1.2}
\begin{cases}
-u''+\lambda u=\kappa(x)|u|^{p-2}u~~&\mbox{on every edge}~ {\rm e} ~\mbox{of}~ \cG,\\
\sum\limits_{\edge \succ {\rm v}}u_{\edge}'({\rm v})=0~~&\mbox{at every vertex}~{\rm v} \in \mathcal{K},
\end{cases}
\end{equation}
where $\edge \succ {\rm v}$ means that the edge $\edge$ is incident at ${\rm v}$, and the notation $u'_{\edge}({\rm v})$ stands for $u'_{\edge}(0)$ or $-u'_{\edge}(\ell_{\edge})$, according to whether the vertex ${\rm v}$ is identified with $0$ or $\ell_{\mathrm{e}}$ (namely, the sum involves the outer derivatives at ${\rm v}$). The second equation is the so-called \emph{Kirchhoff condition}.

Our main result can be stated as follows.
\begin{theorem}\label{thm: main ex}
Let $\cG$ be any noncompact metric graph having a non-empty compact core $\mathcal{K}$, and let $p>6$. Then, for any $\mu>0$, problem \eqref{1.2} with the mass constraint \eqref{mass const} has a positive solution $u$ associated to a Lagrange multiplier $\lambda>0$. This solution corresponds to a critical point of $E(\cdot\,,\mathcal{G})$ on $H_{\mu}^1(\mathcal{G})$ at a strictly positive energy level.
\end{theorem}

Let us present the main ideas to prove  Theorem \ref{thm: main ex}.

It is simple to check that the functional $E(\cdot\,,\cG)$ admits a mountain pass geometry in $H_{\mu}^1(\mathcal{G})$ for any $\mu>0$. However, the existence of a bounded Palais-Smale sequence at the mountain pass level, a key preliminary step to obtain a critical point, is not directly available. To overcome this issue, following the strategy presented in \cite{CJS-2022}, we first  consider the family of functionals $E_{\rho}(\cdot,\mathcal{G}): H^1(\mathcal{G}) \to \mathbb{R}$ given by
\begin{eqnarray}\label{para func}
E_{\rho}(u,\mathcal{G})=\frac{1}{2}\int_{\mathcal{G}}|u'|^2dx-\frac{\rho }{p}\int_{\mathcal{K}}|u|^pdx, \qquad \forall u\in H^1(\mathcal{G}), \, \forall \rho\in \left[\frac12, 1 \right].
\end{eqnarray}
As in \cite{CJS-2022} we shall rely on a  parametrized version of a mountain-pass theorem on a constraint established in \cite{BCJS-2022}. Roughly, this theorem, recalled here as Theorem \ref{thm: monot trick second order}, guarantees, for almost every value of a parameter, the existence of a bounded Palais-Smale sequence with second-order information. 
Applying Theorem \ref{thm: monot trick second order} on the family $E_{\rho}(\cdot,\mathcal{G})$ we deduce that, for almost every $\rho \in \left[\frac12, 1 \right]$, these functionals admit a bounded Palais-Smale sequence $\{u_{\rho,n}\}$ of non negative functions with ``approximate Morse index" information.

The next step is to show that, for such $\rho \in \left[\frac12, 1 \right]$, the sequence $\{u_{\rho,n}\}$ converges to some $u_{\rho}$. If this is the case, as a consequence of the Morse type information on $\{u_{\rho,n}\}$ one can directly deduce that $u_{\rho}$ has a Morse index $m(u_{\rho})$ at most 2, see Definition \ref{def: morse}.
 In \cite{CJS-2022} the convergence of $\{u_{\rho,n}\}$ was immediate due to the compactness of $\cG$.  In our setting, despite the fact that the nonlinearity acts only on the compact core, the non-compactness of the graph prevents us from directly concluding that bounded Palais-Smale sequences converge. This kind of difficulty is regularly encountered  in the search of normalized solutions in $L^2$-supercritical problems when the underlying domain is noncompact, see for example \cite{JeLu20,MeSc22}.  To overcome it one needs to benefit from information on the sign of an associated sequence of {\it almost} Lagrange multipliers. In our problem we derive this information by using the  ``approximate Morse index" information carried by the sequence $\{u_{\rho,n}\}$.  This argument seems new and we suspect that it could be used in other problems. In order to emphasize its generality we present it, in Section \ref{Section:2}, in a general framework.

Having been able to show, for almost every $\rho \in \left[\frac12, 1 \right]$, the convergence of $\{u_{\rho,n}\}$ to some $u_{\rho}$ with Morse index at most 2 (see Proposition \ref{prop: ex for ae rho} for a precise statement), our proof follows the lines of \cite{CJS-2022}. Namely, taking a sequence $\rho_n \to 1^-$, we adapt the blow-up analysis, developed  in \cite{CJS-2022}, to show that $\{u_n\}:=\{u_{\rho_n}\}$ is  bounded. Then, we use some arguments previously developed in the proof of Proposition \ref{prop: ex for ae rho}  to deduce that $\{u_n\}$ converges and to end the proof of Theorem \ref{thm: main ex}.

The rest of the paper is mainly devoted to the proof of Theorem \ref{thm: main ex}. Before proceeding, we would like to discuss some open problems and possible further directions of research.

\begin{remark}
 Since they are obtained through a mountain pass procedure on a codimension $1$ manifold, it is natural to expect that the solutions we found are orbitally unstable. This is what happens in mass-superlinear problems when the underlying equation is autonomous and set on all the space. See, for example, [39] in that direction. However, here things are more complex since we have little information on the variational characterizations of the solutions that we obtain. In comparison to [39], we do not know if our solutions can be characterized as a {\it ground state} or as a global minimizer on some additional constraint. More globally, the issue of orbital stability (and the related issue of the monotonicity of the mass with respect to the Lagrange multiplier) appears complex \emph{on a general non-compact graph}. So far, orbital stability/instability of bound states on metric graphs has only been characterized for some specific examples, see \cite{ACFN-2014, CDS, NPS} and references therein. 
\end{remark}

\begin{remark}
Another challenging issue would consist in characterizing precisely the properties of the solutions in terms of the number and the location of local and global maximum points, at least in the limit of large/small mass parameter. This is also related to multiplicity issues. We refer to \cite{AST2019, BMP} for some recent results in the subcritical case, however it seems unlikely that they are directly adaptable to the supercritical case (the former exploits a doubly constrained minimization argument; the latter uses some techniques specifically suited to deal with the cubic NLS). 

On the other hand, it is not difficult to prove the following basic result, which actually holds for \emph{any} one-signed critical point of $E(\cdot, \cG)$ on $S_\mu$.
\end{remark}

\begin{proposition}\label{prop: properties}
For $p >6$ and $\mu>0$, let $u \in H^1_\mu(\cG)$ be a critical point of $E(\cdot, \cG)$ defined in \eqref{1.1} under the mass constraint \eqref{mass const}, and suppose that $u>0$ and $\lambda>0$. If $\bar x \in \cG$ is a local maximum point of $u$, then $\bar x \in \mathcal{K}$, the compact core. Furthermore, on any half-line $\ell_i \simeq [0,+\infty)$ one has $u(x) = c_i e^{-\sqrt{\lambda}x}$ for some $c_i>0$.
\end{proposition}

The proof follows directly by integrating the equation on each half-line, and using the fact that $\lambda>0$.

\begin{remark}
Finally, we make a short remark on extending our methods to study the existence of bound states in cases where the nonlinearity acts on the full graph, and not only its compact core. In this direction, our min-max geometric argument still works, but some issues arise in studying the compactness of Palais-Smale sequences. Considering what happens in the subcritical and critical cases (see e.g. \cite{AST-CVPDE2015, AST-JFA2016, AST-CMP2017}), it is expected that topological and metric properties of the underlying graph play a role.
\end{remark}

The paper is organized as follows. In Section \ref{Section:2} we recall in Theorem \ref{thm: monot trick second order} the main result of \cite{CJS-2022} and explore some of its consequences. In particular, we show that the second order information on the Palais-Smale sequences can be used to get uniform bounds from below on the associated sequences of almost Lagrange parameter, see Lemma \ref{*lem-Louis1}. In Section \ref{sec: appr} we show that $E_\rho(\cdot\,,\cG)$  has a mountain pass critical point for almost every $\rho \in \left[\frac12, 1 \right]$ and we prove Proposition \ref{prop: ex for ae rho}. Section  \ref{sec:4}  presents the proof of Theorem \ref{thm: main ex}.


\section{Some preliminary definitions and results}\label{Section:2}

In this section, we recall \cite[Theorem 1]{BCJS-2022} and explore some of its consequences.

Let $\left(E,\langle \cdot, \cdot \rangle\right)$ and $\left(H,(\cdot,\cdot)\right)$ be two \emph{infinite-dimensional} Hilbert spaces and assume that $E\hookrightarrow H \hookrightarrow E'$,
with continuous injections.  For simplicity, we assume that the continuous injection $E\hookrightarrow H$ has norm at most $1$ and identify $E$ with its image in $H$.
Set \[ \begin{cases} \|u\|^2=\langle u,u \rangle,\\ |u|^2=(u,u),\end{cases}\quad u\in E,\]
and we define for $\mu>0$: \[S_\mu= \{ u \in E\big|\, |u|^2=\mu \}. \]
In the context of this paper, we shall have $E=H^1(\cG)$ and $H=L^2(\cG)$.
Clearly, $S_{\mu}$ is a smooth submanifold of $E$ of codimension $1$. Its tangent space at a given point $u \in S_{\mu}$ can be considered as the closed codimension $1$ subspace of $E$ given by:
\[  T_u S_{\mu}= \{v \in E \big|\, (u,v) =0 \}.\]
In the following definition, we denote by $\|\cdot\|_*$ and $\|\cdot\|_{**}$, respectively, the operator norm of $\mathcal{L}(E,\R)$ and of $\mathcal{L}(E,\mathcal{L}(E,\R))$.

\begin{definition}
Let $\phi : E \rightarrow \mathbb{R}$ be a $C^2$-functional on $E$ and $\alpha \in (0,1]$. We say that $\phi'$ and $\phi''$ are $\alpha$-H\"older continuous on bounded sets if for any $R>0$ one can find $M=M(R)>0$ such that, for any $u_1,u_2\in B(0,R)$:
\begin{equation}\label{Holder}
||\phi'(u_1)-\phi'(u_2)||_*\leq M ||u_1-u_2||^{\alpha}, \quad ||\phi''(u_1)-\phi''(u_2)||_{**} \leq M||u_1-u_2||^\alpha.
\end{equation}
\end{definition}

\begin{definition}\label{def D}
Let $\phi$ be a $C^2$-functional on $E$. For any $u\in E$, we define the continuous bilinear map:
\[ D^2\phi(u)=\phi''(u) -\frac{\phi'(u)\cdot u}{|u|^2}(\cdot,\cdot). \]
\end{definition}

\begin{definition}\label{def: app morse}
Let $\phi$ be a $C^2$-functional on $E$. For any $u\in \M$ and $\theta >0$, we define
 the \emph{approximate Morse index} by
\[
\tilde m_\theta(u)= \sup \left\{\dim\,L\left| \begin{array}{l} \ L \text{ is a subspace of $T_u \M$ such that: }
D^2\phi(u)[\varphi, \varphi]<-\theta \|\varphi\|^2, \quad \forall \varphi \in L \setminus \{0\} \end{array}\right.\right\}.
\]
If $u$ is a critical point for the constrained functional $\phi|_{\M}$ and $\theta=0$, we say that this is the \emph{Morse index of $u$ as constrained critical point}.
\end{definition}

The following abstract theorem was established in \cite{BCJS-2022}. Its derivation is based on a combination of ideas from \cite{FG-1992,FG-1994,J-PRSE1999} implemented in a convenient geometric setting. Note that related theorems, but dealing with unconstrained functionals, were recently developed in \cite{BeRu,LoMaRu}.

\begin{theorem}[Theorem 1 in \cite{BCJS-2022}]\label{thm: monot trick second order}
Let $I \subset (0,+\infty)$ be an interval and consider a family of $C^2$ functionals $\Phi_\rho: E \to \mathbb{R}$ of the form
\[
\Phi_\rho(u) = A(u) -\rho B(u), \qquad \rho \in I,
\]
where $B(u) \ge 0$ for every $u \in E$, and
\begin{equation}\label{hp coer}
\text{either $A(u) \to +\infty$~ or $B(u) \to +\infty$ ~ as $u \in E$ and $\|u\| \to +\infty$.}
\end{equation}
Suppose moreover that $\Phi_\rho'$ and $\Phi_\rho''$ are $\alpha$-H\"older continuous on bounded sets for some $\alpha \in (0,1]$.
Finally, suppose that there exist $w_1, w_2 \in \M$ (independent of $\rho$) such that, setting
\[
\Gamma= \left\{ \gamma \in C([0,1],\M)\big| \ \gamma(0) = w_1, \quad \gamma(1) = w_2\right\},
\]
we have
\begin{equation}\label{mp geom}
c_\rho:= \inf_{\gamma \in \Gamma} \ \max_{ t \in [0,1]} \Phi_\rho(\gamma(t)) > \max\{\Phi_\rho(w_1), \Phi_\rho(w_2)\}, \quad \rho \in I.
\end{equation}
Then, for almost every $\rho \in I$, there exist sequences $\{u_n\} \subset \M$ and $\zeta_n \to 0^+$ such that, as $n \to + \infty$,
\begin{itemize}
\item[(i)] $\Phi_\rho(u_n) \to c_\rho$;
\item[(ii)] $||\Phi'_\rho|_{\M}(u_n)||_* \to 0$;
\item[(iii)] $\{u_n\}$ is bounded in $E$;
\item[(iv)] $\tilde m_{\zeta_n}(u_n) \le 1$.
\end{itemize}
\end{theorem} 
\noindent
It is easily observed, see \cite[Remarks 1.3]{BCJS-2022}, from Theorem \ref{thm: monot trick second order}  (ii)-(iii) that 
\begin{equation}\label{free-gradient}
 \Phi'_{\rho}(u_n) + \displaystyle \lambda_n (u_n, \cdot) \to 0 \, \mbox{ in }  E' \mbox{ as } n \to + \infty
\end{equation} where we have set 
\begin{equation}\label{def-almost-Lagrange} 
\lambda_n := - \displaystyle \frac{1}{\mu}( \Phi'_\rho(u_n)\cdot u_n).
\end{equation} 
Also, Theorem \ref{thm: monot trick second order} (iv) directly implies that if there exists a subspace $W_n \subset T_{u_n}\M$ such that
\begin{equation}\label{L-Hessian}
D^2\Phi_{\rho}(u_n)[w,w] = \Phi''_{\rho}(u_n)[w,w] + \displaystyle \lambda_n (w,w) < - \zeta_n ||w||^2, \quad \mbox{for all } w \in W_n \setminus \{0\},
\end{equation}
then necessarily $\dim W_n \leq 1$. We called the sequence $\{\lambda_n\} \subset \R$ defined in \eqref{def-almost-Lagrange} the sequence of \emph{almost Lagrange multipliers}.

We shall see in Section \ref{sec: appr} that if $\{u_n\} \subset \M$  converges to some $u \in \M$ then information on the Morse index of $u$ as constrained critical point can be obtained. First, let us show how information on the sequence $\{\lambda_n\} \subset \R$ can be derived.
 
\begin{lemma}\label{*lem-Louis1}
Let $\{u_{n}\} \subset S_{\mu}$, $\{\lambda_{n}\} \subset \mathbb{R}$ and $\{\zeta_n\}\subset \mathbb{R}^+$ with $\zeta_n\to 0^+$. Assume that the following conditions hold:
\begin{itemize}
\item[(i)] For large enough $n \in \mathbb{N}$, all subspaces $W_n \subset E$ with the property
\begin{equation}\label{LL-Hess crit*}
\Phi''_{\rho}(u_n)[\varphi,\varphi]+\lambda_n|\varphi|^2 < -\zeta_n \|\varphi\|^2, \qquad \mbox{for all } \, \varphi \in W_n \backslash \{0\},
\end{equation}
satisfy: $\dim W_n \leq 2$.
\item[(ii)] There exist $\lambda \in \mathbb{R}$, a subspace $Y$ of $E$ with $\dim Y \geq 3$ and $a>0$ such that, for large enough $n \in \mathbb{N}$,
\begin{equation}\label{L1*}
\Phi''_{\rho}(u_n)[\varphi,\varphi] + \lambda |\varphi|^2 \leq -a ||\varphi||^2,  \qquad \mbox{for all } \, \varphi \in Y.
\end{equation}
\end{itemize} 
 Then $\lambda_n > \lambda$ for all large enough $n \in \mathbb{N}$. In particular, if (\ref{L1*}) holds for any $\lambda <0$, then $\liminf_{n \to \infty} \lambda_n \geq 0$.
\end{lemma}
\begin{proof}
Suppose by contradiction that $\lambda_n \leq \lambda$ along a subsequence. From \eqref{L1*} we have, for all $\varphi \in Y \backslash \{0\},$
\begin{equation*}
\begin{split}
\Phi''_{\rho}(u_n)[\varphi,\varphi]+\lambda_n|\varphi|^2
& = \Phi''_{\rho}(u_n)[\varphi,\varphi]+\lambda|\varphi|^2 + (\lambda_n - \lambda)|\varphi|^2 \\
& \leq -a ||\varphi||^2 + (\lambda_n - \lambda)|\varphi|^2.
\end{split}
\end{equation*}
Now, since $\zeta_n\to 0^+$, there exists $n_0 \in \mathbb{N}$ such that: $\forall n \geq n_0$,
$\zeta_n <a$. Thus, for an arbitrary $n \geq n_0$, we obtain 
$$ \Phi''_{\rho}(u_n)[\varphi,\varphi]+\lambda_n|\varphi|^2 \leq  -\zeta_n \|\varphi\|^2, \qquad \mbox{for all } \, \varphi \in Y \backslash \{0\},$$
in contradiction with \eqref{LL-Hess crit*} since dim  $Y >2$.
\end{proof}

\section{Mountain pass solutions for approximating problems}\label{sec: appr}

Theorem \ref{thm: monot trick second order} will be a key ingredient in the proof of Theorem \ref{thm: main ex}. To show that the family  $E_\rho(\cdot\,\cG)$ enters into the framework of Theorem \ref{thm: monot trick second order} we first need to show that it has a mountain pass geometry on $H^1_\mu(\cG)$ uniformly with respect to $\rho\in \left[\frac12, 1 \right]$.


\begin{lemma}\label{MP-geometry}
 For every $\mu>0$, there exist $w_1, w_2 \in S_\mu$ independent of $\rho \in \left[\frac12,1\right]$ such that
\begin{eqnarray*}
c_{\rho}:=\inf\limits_{\gamma\in
\Gamma} \, \max\limits_{t\in[0,1]}E_{\rho}(\gamma(t),\mathcal{G}) > \max\{E_{\rho}(w_1,\mathcal{G}), E_{\rho}(w_2,\mathcal{G})\}, \qquad \forall
\rho\in \left[\frac{1}{2},1\right],
\end{eqnarray*}
where
$\Gamma:=\left\{\gamma\in C([0,1], H^1_{\mu}(\mathcal{G}))\big| ~\gamma~ \mbox{is continuous},~\gamma(0)=w_1,\gamma(1)=w_2\right\}.$
\end{lemma}
\begin{proof}
For any $\mu, k>0$, denote
\begin{eqnarray*}
&A_{\mu,k}:=\{u\in H_{\mu}^1(\mathcal{G})\big| \int_{\mathcal{G}}|u'|^2dx<k\},\\
&\partial A_{\mu,k}:=\{u\in H_{\mu}^1(\mathcal{G})\big| \int_{\mathcal{G}}|u'|^2dx=k\}.
\end{eqnarray*}
First note that, since the graph is noncompact, for any $\mu, k>0$, we have $A_{\mu,k} \neq \emptyset$ (and, similarly, $A_{\mu,k} \neq \emptyset$). Indeed, take any function $v \in C^\infty_c(\R)$ with $\|v\|_{L^2(\R)}^2 =\mu$, and consider $v_t(x) = t^{\frac12}v(tx)$, for $t>0$. We have that $\|v_t\|_{L^2(\R)}^2 = \mu$ and $\|v_t'\|_{L^2(\R)}^2 = t^2\|v'\|_{L^2(\R)}^2$ for every $t>0$. Since $\cG$ contains at least a half-line, it contains arbitrarily large intervals. Therefore, any function $v_t$ can be regarded as a function in $H^1_\mu(\cG)$, still denoted by $v_t$, with the support entirely located on the half-line $l_1$; and, in particular, $v_t \in A_{\mu,k}$ for $t$ small enough.

Now, by the Gagliardo-Nirenberg inequality on metric graphs (see Proposition 2.1 in \cite{AST-JFA2016}), we obtain
\begin{align}\label{11-21-1} E_{\rho}(u,\mathcal{G})&\ge\frac{1}{2}\|u'\|_{L^2(\mathcal{G})}^2- \rho\frac{C_p}{p}\mu^{\frac{p}{4}+\frac{1}{2}}\|u'\|_{L^2(\mathcal{G})}^{\frac{p}{2}-1}, \qquad \forall u\in H^1_{\mu}(\cG).
\end{align}
Then, for any $\mu>0$ and $u\in \partial A_{\mu, k_0}$ with $ k_0 = \frac{1}{2}(\frac{p}{2C_p})^{\frac{4}{p-6}}\mu^{-\frac{p+2}{p-6}}$, we have
\begin{eqnarray}\label{3-7-2}
\inf\limits_{u\in \partial A_{\mu,k_0}}E_{\rho}(u,\mathcal{G})\ge k_0\left(\frac12-\frac{C_p}{p} \mu^{\frac{p+2}{4}} k_0^{\frac{p-6}{4}}\right) =: \alpha>0,
\end{eqnarray}
for every $\rho \in [1/2,1]$.

Next, observe that for any $u \in A_{\mu,k}$ 
\begin{eqnarray*}
E_{\rho}(u,\mathcal{G})\le\frac{1}{2}\int_{\mathcal{G}}|u'|^2\,dx\le\frac{1}{2}k.
\end{eqnarray*}
Thus, recalling that $A_{k,\mu} \neq \emptyset$ for all $\mu, k >0$, it is possible to choose a $w_1 \in H^1_\mu(\cG)$ such that
\begin{equation}\label{mpw1}
\|w_1'\|_{L^2(\cG)}^2 <k_0 \quad \text{and} \quad E_\rho(w_1, \cG) \le \frac{\alpha}2 \quad \forall \rho \in \left[\frac12,1\right].
\end{equation}
Moreover, we also observe that we can identify any bounded edge, say ${\mathrm e_1}$, with the interval $[-\ell_{1}/2, \ell_{1}/2]$. It follows that any compactly supported $H^1$ function $w$ on such interval, with mass $\mu$, can be seen as a function in $H^1_\mu(\cG)$. Defining $w_t(x) := t^{1/2} w(t x)$, with $t >1$, it is not difficult to check that $w_t \in H^1_\mu(\cG)$ (notice in particular that the support of $w_t$ is shrinking on ${\rm{e}_1}$ as $t$ becomes larger),
and that
\begin{align*}
E_\rho(w_t, \cG) &= \frac{t^2}{2} \int_{{\mathrm e_1}} |w'|^2\,dx- \frac{\rho t^{\frac{p-2}{2}}}{p}\int_{{\mathrm e_1}} |w|^p\,dx  \le  \frac{t^2}2\left( \int_{{\mathrm e_1}} |w'|^2\,dx - \frac{t^{\frac{p-6}{2}}}{p}\int_{{\mathrm e_1}} |w|^p\,dx\right),
\end{align*}
for every $\rho \in [\frac{1}{2}, 1]$.
Since $p>6$, then the right-hand side tends to $-\infty$ as $t \to +\infty$, and in particular there exists $t_2>0$ large enough such that $\|w'_t\|^2_{L^2(\mathcal{G})}=t^2 \|w'\|_{L^2(\cG)}^2>2k_0$ and $E_{\rho}(w_{t},\mathcal{G})<0$ for all $t> t_2$ and $\rho\in[\frac{1}{2}, 1]$. We set $w_2:=w_{2t_2}$, and we point out that
\begin{equation}\label{mpw2}
\|w_2'\|_{L^2(\cG)}^2 >2k_0 \quad \text{and} \quad E_\rho(w_2, \cG) <0\quad \forall \rho \in \left[\frac12,1\right].
\end{equation}
At this point the thesis follows easily. Let $\Gamma$ and $c_\rho$ be defined as in the statement of the lemma for our choice of $w_1$ and $w_2$; the fact that $\Gamma \neq 0$ is straightforward, since 
\[
\gamma_0(t):=\frac{\mu^{1/2}}{\| (1-t) w_1 + t w_2\|_{L^2(\cG)}}  \left[(1-t) w_1 + t w_2\right] \qquad t \in [0,1]
\]
belongs to $\Gamma$. By \eqref{mpw1} and \eqref{mpw2}, we note that for every $\gamma \in \Gamma$ there exists $t_\gamma \in [0,1]$ such that $\gamma(t_\gamma) \in \pa A_{\mu, k_0}$, by continuity. Therefore, for any $\rho \in [1/2,1]$, we have for every $\gamma \in \Gamma$ 
\[
\max_{t \in [0,1]} E_\rho(\gamma(t),\cG) \ge E_\rho(\gamma(t_\gamma),\cG) \ge \inf_{u \in \pa A_{\mu, k_0}} E_\rho(u, \cG) \geq  \alpha
\]
(see \eqref{3-7-2}) and thus $c_\rho \ge \alpha$  while
\[
\max\{E_\rho(w_1,\cG), E_\rho(w_2,\cG)\} = E_\rho(w_1,\cG) < \frac{\alpha}2. \qedhere
\] 
\end{proof}

The following lemma, together with Lemma \ref{*lem-Louis1},  will be central to establish the convergence of the Palais-Smale sequences provided by the application of Theorem \ref{thm: monot trick second order}.
\begin{lemma}\label{L-eigenvalue}
For any $\lambda <0$,  there exists a subspace $Y$ of $H^1(\mathcal{G})$ with dim $Y=3$ such that
\begin{equation}\label{9-12-1}
\int_{\mathcal{G}}|w'|^2 \,dx+\lambda \int_{\mathcal{G}}|w|^2 \,dx \leq \frac{\lambda}{2} \|w\|_{H^1(\cG)}^2,  \qquad \forall \, w \in Y.
\end{equation}
\end{lemma}
\begin{proof}
Take $\phi\in C_0^{\infty}(\mathbb{R}^+)$ with ${\rm{supp}}\,\phi \subset [\frac{3}{4}, \frac{5}{4}]$ such that $\int_0^{+\infty} |\phi|^2dx=1$. Recalling that $l_1$ is an arbitrary half-line of $\cG$, we define the following functions for $\sigma>0$:
\begin{eqnarray*}\psi_{0,\sigma}(x)= \left\{
\begin{array}{ll}
0~~& \mbox{for}~x\in \mathcal{\cG} \setminus l_1\\
\sigma^{\frac{1}{2}}\phi(\sigma x)~~&~\mbox{for}~x\in l_1,
\end{array}
\right.
\end{eqnarray*}
\begin{eqnarray*}\psi_{1,\sigma}(x)= \left\{
\begin{array}{ll}
0~~& \mbox{for}~x\in \mathcal{\cG} \setminus l_1\\
\sigma^{\frac{1}{2}}\phi(\sigma x -1 )~~&~\mbox{for}~x\in l_1,
\end{array}
\right.
\end{eqnarray*}
and
\begin{eqnarray*}\psi_{2,\sigma}(x)= \left\{
\begin{array}{ll}
0~~& \mbox{for}~x\in \mathcal{\cG} \setminus l_1\\
\sigma^{\frac{1}{2}}\phi(\sigma x -2 )~~&~\mbox{for}~x\in l_1.
\end{array}
\right.
\end{eqnarray*}
Clearly, for any $\sigma>0$, $\int_{\mathcal{G}} |\psi_{j,\sigma}|^2\,dx=1$ for all $j=0,1,2$, and the three functions are mutually orthogonal in $H^1(\mathcal{G})$, having disjoints supports.  Assume that $w=\sum\limits_{j=0}^2\theta_j\psi_{j,\sigma}$. Then
\[
\begin{split}
\int_{\mathcal{G}}|w'|^2 \, dx+\lambda \int_{\mathcal{G}}|w|^2 \,dx
& = \sigma^2 \Big( \sum_{j=0}^2 \theta_j^2 \int_0^{+\infty} |\phi'|^2\,dx \Big) + \lambda \Big( \sum_{j=0}^2 \theta_j^2 \int_0^{+\infty} |\phi|^2\,dx \Big) \\
& = (\sigma^2 \alpha + \lambda) \sum_{j=0}^2 \theta_j^2,
\end{split}
\]
where $\alpha:= \int_0^{+\infty} |\phi'|^2\,dx>0$. Similarly, $\|w\|_{H^1(\cG)}^2 = (\sigma^2 \alpha + 1) \sum_{j=0}^2 \theta_j^2$. Therefore, provided that $w \neq 0$,
\[
\frac{\int_{\mathcal{G}}|w'|^2 \, dx+\lambda \int_{\mathcal{G}}|w|^2 \,dx}{\|w\|_{H^1(\cG)}^2} = \frac{\sigma^2 \alpha + \lambda}{\sigma^2 \alpha + 1} \le \frac{\lambda}{2}
\]
for every $\sigma>0$ sufficiently small, and for every $\theta_0, \theta_1, \theta_2 \in \R$.
\end{proof}

\begin{remark} The validity of the previous lemmata relies on the existence of at least one half-line, thus on the non-compact character of the graph. Results like Lemma \ref{L-eigenvalue} are expected to hold broadly in problems set on a non-compact domain.
\end{remark}
Finally, we recall a definition. In what follows, $\chi_A$ denotes the characteristic function of $A$.
 \begin{definition}\label{def: morse}
For any graph $\mathcal{F}$, any subgraph $\mathcal{F'} \subset \mathcal{F}$, and any solution $U \in C(\mathcal{F}) \cap H^1_{{\rm loc}}(\mathcal{F})$, not necessarily in $H^1(\mathcal{F})$, of
\begin{equation}\label{eq U}
\begin{cases}
-U'' + \lambda U = \rho |U|^{p-2} U \chi_{\mathcal{F}'} & \text{in $\mathcal{F}$}, \\
\sum_{{\rm e} \succ {\rm v}} U'({\rm v}) = 0 &  \text{for any vertex ${\rm v}$ of $\mathcal{F}$},
\end{cases}
\end{equation}
with $\lambda, \rho \in \R$, we consider
\begin{equation}\label{second differential}
Q(\varphi;U, \mathcal{F}):= \int_{\mathcal{F}} \left(|\varphi'|^2 + (\lambda-(p-1)\rho|U|^{p-2} \chi_{\mathcal{F'}}) \varphi^2\right)\,dx, \quad \forall \varphi \in H^1(\mathcal{F}) \cap C_c(\mathcal{F}).
\end{equation}
The \emph{Morse index of $U$}, denoted by $m(U)$, is the maximal dimension of a subspace $W \subset H^1(\mathcal{F}) \cap C_c(\mathcal{F})$ such that $Q(\varphi; U, \mathcal{F})<0$ for all $\varphi \in W \setminus \{0\}$.
\end{definition}


The main aim of this section is to prove the following result.


\begin{proposition}\label{prop: ex for ae rho}
For any fixed $\mu>0$ and almost every $\rho \in I_{\rho}$, there exists $(u_\rho, \lambda_{\rho})\in H^1_\mu(\cG)\times \mathbb{R}^+$ which solves
\begin{equation}\label{pb rho}
\begin{cases}
-u_\rho'' + \lambda_\rho u_\rho = \rho \kappa(x) u_\rho^{p-1}, \quad u_\rho>0 & \text{in $\mathcal{\cG}$}, \\
\sum_{{\rm e} \succ {\rm v}} u_\rho'({\rm v}) = 0 &  \text{for any vertex ${\rm v}$}.
\end{cases}
\end{equation}
 Moreover, $E_{\rho}(u_{\rho},\cG)=c_{\rho}$ and $m(u_\rho) \le 2$.
\end{proposition}

\begin{proof}
For simplicity, we omit the dependence of the functionals $E_\rho(\cdot\,\cG)$ on $\cG$.
We apply Theorem \ref{thm: monot trick second order} to the family of functionals $E_\rho$, with $E=H^1(\cG)$, $H=L^2(\cG)$, $S_\mu = H^1_\mu(\cG)$, and
$\Gamma$ defined in Lemma \ref{MP-geometry}. Setting
\[
A(u) = \frac12\int_{\cG} |u'|^2\,dx \quad \text{and} \quad  B(u) = \frac{\rho}{p}\int_{\cG} |u|^p\,dx,
\]
assumption \eqref{hp coer} holds, since we have that
\[
u \in H^1_\mu(\cG), \ \|u\| \to +\infty \quad \implies \quad A(u) \to +\infty.
\]
Let $E'_{\rho}$ and $E''_{\rho}$ denote respectively the free first and second derivatives of $E_{\rho}$. Clearly, $E'_{\rho}$ and $E''_{\rho}$ are both of class $C^1$, and hence locally H\"older continuous, on $H_\mu^1(\cG)$, which implies that assumption \eqref{Holder} holds. \medskip

Thus, taking into account Lemma \ref{MP-geometry}, by Theorem \ref{thm: monot trick second order}  and the considerations just after it, for almost every $\rho \in [1/2,1]$, there exist a bounded sequence $\{u_{n, \rho}\} \subset H_\mu^1(\cG)$, that we shall denote simply by $\{u_n\}$ from now on, and a sequence $\{\zeta_n\}\subset \mathbb{R}^+$ with $\zeta_n\to 0^+$, such that
\begin{equation}\label{const crit}
E'_{\rho}(u_n) +  \lambda_n u_n \to 0 \quad \mbox{in the dual of} \quad H^1_{\mu}(\mathcal{G}),
\end{equation}
where
\begin{equation}\label{lambda}
\lambda_n:= -\frac{1}{\mu} E'_\rho(u_n) u_n.
\end{equation}
Moreover, if the inequality
\begin{eqnarray}\label{L-Hess crit}
\int_{\cG} \left[ |\varphi'|^2 + \left( \lambda_n -(p-1)\rho \kappa(x)|u_n|^{p-2}\right) \varphi^2\right]  \, dx = E''_\rho(u_n)[\varphi, \varphi] + \lambda_n ||\varphi||_{L^2(\mathcal{G})}^2 < -\zeta_n \|\varphi\|_{H^1(\cG)}^2
\end{eqnarray}
holds for any $\varphi \in W_n \setminus \{0\}$ in a subspace $W_n$ of $T_{u_n} S_\mu$, then the dimension of $W_n$ is at most $1$.
In addition, as explained in \cite[Remark 1.4]{BCJS-2022}, since $u \in H^1_\mu(\cG) \ \Longrightarrow  |u| \in H^1_\mu(\cG)$, $w_1, w_2 \ge 0$, the map $u \mapsto |u|$ is continuous, and $E_\rho(u) = E_\rho(|u|)$, it is possible to choose $\{u_n\}$ with the property that $u_n \ge 0$ on $\cG$. \medskip

The sequence $\{u_n\}$ being bounded,  it follows by (\ref{lambda}) that $\{\lambda_n\}$ is bounded. Then, passing to a subsequence, there exists $\lambda_\rho\in \mathbb{R}$ such that $\lim\limits_{n\to+\infty}\lambda_n=\lambda_{\rho}$.
Furthermore, there exists $u_\rho\in H^1(\mathcal{G})$ such that
 \begin{eqnarray}
 &&u_{n}\rightharpoonup u_{\rho}~~\mbox{in}~~ H^1(\mathcal{G}),\label{3-5-1}\\
 &&u_n\to u_{\rho}~~\mbox{in}~~L_{{\rm{loc}}}^r(\mathcal{G}), r>2,\label{3-5-2}\\
 &&u_{n}(x)\to u_{\rho}(x)~~\mbox{for a.e.}~ x\in \mathcal{G},\label{3-5-3}
 \end{eqnarray}
which implies that $u_{\rho}\ge0$. Also, recalling (\ref{const crit}) and that $\lim\limits_{n\to+\infty}\lambda_n=\lambda_{\rho}$, we see that $u_{\rho}$ satisfies
\begin{equation}\label{e1}
- u_{\rho}'' + \lambda_{\rho} u_{\rho}=\rho\kappa(x) u_{\rho}^{p-1}
\end{equation}
with the Kirchhoff condition at the vertices. In order to show that $u_{\rho} \not \equiv 0$ we shall first prove that $\lambda_{\rho}\ge0$. Here the Morse type information \eqref{L-Hess crit} proves decisive. Firstly, since the codimension of $T_{u_n} S_\mu$ is $1$, we infer that if the inequality \eqref{L-Hess crit} holds for every $\varphi \in V_n \setminus \{0\}$ for a subspace $V_n$ of $H^1(\cG)$, then the dimension of $V_n$ is at most $2$. Secondly, by Lemma \ref{L-eigenvalue}, we have that for every $\lambda<0$, there exist a subspace $Y$ of $E$ with $\dim Y \geq 3$ and $a>0$ such that, for $n \in \mathbb{N}$ large,
\[
\begin{split}
E''_{\rho}(u_n)[\varphi,\varphi] + \lambda |\varphi|^2 &\leq \int_{\mathcal{G}}|\varphi'|^2 \,dx+\lambda \int_{\mathcal{G}}|\varphi|^2 \le -a \|\varphi\|^2,  \qquad \mbox{for all } \, \varphi \in Y \backslash \{0\}.
\end{split}
\]
Therefore, Lemma \ref{*lem-Louis1} implies that $\lambda_\rho \ge 0$.

\medskip

 Now from (\ref{const crit}) and the fact that $\lambda_n \to \lambda_\rho$, we deduce that
\[
\int_{\cG} \left(u_n' \varphi'+\lambda_\rho  u_n \varphi\right) \,dx - \rho \int_{\mathcal{K}} u_n^{p-1} \varphi\,dx =o(1) \|\varphi\|_{H^1(\cG)}.
\]
Moreover, by \eqref{e1}, 
\[
\int_{\cG} \left(u_\rho' \varphi'+\lambda_\rho  u_\rho \varphi\right) \,dx - \rho \int_{\mathcal{K}} u_\rho^{p-1} \varphi\,dx =0
\]
Therefore, taking the difference and exploiting \eqref{3-5-1}-\eqref{3-5-3} we infer that
\begin{equation}\label{C-PS'}
\int_{\mathcal{G}}|(u_n-u_\rho)'|^2\, dx +\lambda_{\rho} \int_{\mathcal{G}}| u_n-u_\rho|^2 \,dx \to 0
\end{equation}
as $n \to \infty$. If we assume that $u_{\rho}  \equiv 0$ then we deduce that 
\begin{equation}\label{e3}
\int_{\mathcal{G}} |u_n'|^2 \,dx + \lambda_{\rho} \int_{\mathcal{G}} |u_n|^2 \,dx   \to 0.
\end{equation}
If $\lambda_{\rho} >0$, this is not possible since $\|u_n\|_{L^2(\cG)}^2=\mu>0$. If $\lambda_{\rho} =0$, then (\ref{e3}) contradicts the fact that the mountain-pass level satisfies $c_{\rho} >0$. Hence, $u_{\rho} \not \equiv 0$. Appealing to the Kirchhoff condition and the uniqueness theorem for ODEs, we have in fact that $u_\rho >0$ in $\cG$: indeed, assume by contradiction that there exists $x_0\in \mathcal{G}$ such that $u_{\rho}(x_0)=0$. If $x_0$ stays in the interior of some edge, then by $u_{\rho}\ge0$ on $\mathcal{G}$ it follows that $u'_{\rho}(x_0)=0$. If instead $x_0$ is a vertex, then by $u_{\rho}(x)\ge0$ and by the Kirchhoff condition we also get $u'_{\rho}(x_0)=0$. Hence, by uniqueness, $u_{\rho}\equiv 0$ on any edge containing $x_0$; but then, by repeating this argument a finite number of times (since $\cG$ has finitely many vertices and edges), we deduce that $u_\rho \equiv 0$ on $\cG$, which is the desired contradiction.

Now we claim that $\lambda_{\rho}>0$. If not, then $\lambda_{\rho}=0$. Identifying $l_1$ with $[0, +\infty)$, we have that $u_{\rho}$ is a $C^2$ function on $[0,+\infty)$ such that $u_{\rho}''(x)=0$ for all $x\in(0,+\infty)$. It follows that $u_\rho > 0$ on $[0,+\infty)$, 
\[
u_\rho(x) = u_\rho(0) + u_\rho'(0)x, \quad \text{and } \quad u_\rho \in L^2(\cG),
\]
which are incompatible. Hence the claim holds. %
%

Having proved that $\lambda_{\rho} >0$, we get from (\ref{C-PS'}) that $u_n \to u_{\rho}$ strongly in $H^1(\mathcal{G})$.

It remains to show that the Morse index $m(u_\rho)$ is at most $2$.  Since the tangent space $T_{u} H^1_\mu(\cG)$ has codimension 1, it suffices to show that $u_{\rho} \in \M$ has Morse index at most $1$ as a constrained critical point.
If not, in view of Definition \ref{def: app morse} we may assume by contradiction that there exists a $W_0 \subset T_{u}\M$ with $\dim W_0 =2$ such that
\begin{equation}\label{10-24-1}
D^2 E_{\rho}(u_{\rho})[w,w]<0 \quad \mbox{for all } w \in W_0 \backslash \{0\}.
\end{equation}
Then, since $W_0$ is of finite dimension, there exists $\beta>0$ such that
$$D^2 E_{\rho}(u_{\rho})[w,w]<-\beta \quad \mbox{for all } w \in W_0 \backslash \{0\}~~\mbox{with}~~\|w\|_{H^1(\mathcal{G})}=1,$$
using the homogeneity of $D^2 E_{\rho}(u_{\rho})$, we deduce that
\begin{equation*}
D^2 E_{\rho}(u_{\rho})[w,w]< -\beta||w||_{H^1(\mathcal{G})}^2 \quad \mbox{for all } w \in W_0 \backslash \{0\}.
\end{equation*}
Now, from \cite[Corollary 1]{BCJS-2022} or using directly that $E_\rho'$ and $E_\rho''$ are $\alpha$-H\"older continuous on bounded sets for some $\alpha \in (0,1]$, 
it follows that there exists $\delta_1>0$ small enough such that, for any $v\in\M$ satisfying $||v-u|| \leq \delta_1$,
\begin{equation}\label{L-conditionD2step1}
D^2E_{\rho}(v)[w,w]<-\frac{\beta}{2}||w||_{H^1(\mathcal{G})}^2\quad \mbox{for all } w \in W_0 \backslash \{0\}.
\end{equation}
Hence, noting that $||u_n-u_{\rho}||_{H^1(\mathcal{G})}\leq\delta_1$ for $n\in \mathbb{N}$ large enough, by \eqref{10-24-1}-\eqref{L-conditionD2step1} and $\zeta_n\to 0^+$ we get
\begin{equation}\label{10-24-2}
D^2E_{\rho}(u_n)[w,w]<-\frac{\beta}{2}||w||_{H^1(\mathcal{G})}^2<-\zeta_n||w||_{H^1(\mathcal{G})}^2\quad \mbox{for all } w \in W_0 \backslash \{0\}
\end{equation}
for any such large $n$. Therefore, since $\dim W_0>1$ and recalling that
$$E''_{\rho}(u_n)[w,w] + \lambda_{\rho}||w||_{L^2(\mathcal{G})}^2 = D^2E_{\rho}(u_n)[w,w],$$
 Equation \eqref{10-24-2} provides a contradiction with Equation \eqref{L-Hess crit}. 
Thus we infer that $\tilde{m}_0(u_\rho) \le 1$. 
\end{proof}

\section{Proof of Theorem \ref{thm: main ex} }\label{sec:4}

From Proposition \ref{prop: ex for ae rho} we know that there exists a sequence $\rho_n \to 1^-$, and a corresponding sequence of critical points $u_{\rho_n} \in H^1_\mu(\cG)$ of $E_{\rho_n}(\cdot\,, \cG)$ constrained to $H^1_\mu(\cG)$, lying at the mountain-pass level $c_{\rho_n}$, with a Morse index $m(u_{\rho_n})$ lesser or equal to $2$. Additionally, the associated Lagrange multipliers satisfy $\lambda_n >0$. \\
To obtain the existence of solutions for the original problem, at a strictly positive energy level, it clearly suffices to show that $\{u_{\rho_n}\}$ converges. In this direction the key point is to show that 
$\{u_{\rho_n}\}$ is bounded in $H^1(\cG)$.  

By Lemma \ref{MP-geometry} and the monotonicity of $c_{\rho}$,
\[
E_1(w_1, \cG) \le E_\rho(w_1, \cG) \le c_{\rho} \le c_{1/2}, \qquad \forall \rho \in \left[\frac12, 1 \right],
\]
which implies that $c_{\rho_n}$ is bounded. In addition, thanks to the Kirchhoff condition
\[
\int_{\cG} \left(|u_n'|^2 + \lambda_n u_n^2\right)dx = \int_{\mathcal{K}} |u_n|^p\,dx,
\]
whence it follows that
\begin{equation*}
\left( \frac12-\frac1p\right) \int_{\cG} |u_{n}'|^2\,dx = c_n + \frac{\lambda_n \mu}p.
\end{equation*}
Therefore, if $\{\lambda_n\}$ is bounded, then $\{u_{\rho_n}\}$ is bounded as well. We shall thus assume that  $\lambda_n \to +\infty$, up to extraction of a subsequence and, in order to reach a contradiction, develop a blow-up analysis for $\{u_{\rho_n}\}$ as in \cite{CJS-2022} (see also \cite{Espetal,EspPet,PieVer}).

Precisely, writing for simplicity $u_n := u_{\rho_n}$ for $\rho_n\to 1^-$, we consider the behavior of a sequence of solutions $\{u_n\} \in H^1(\cG)$ of the following problem:
\begin{equation}\label{eq: blow-up}
\begin{cases}
-u_n''+\lambda_n u_n = \rho_n \kappa(x) u_n^{p-1} & \text{on $\cG$}, \\
u_n>0 & \text{on $\cG$}, \\
\sum_{\rm{e} \succ \rm{v}} u_{e,n}'(\rm{v}) = 0& \forall \rm{v} \in \mathcal{V},
\end{cases}
\end{equation}
where $m(u_n) \leq 2$ for all $n  \in \mathbb{N}$, and it is assumed that $\lambda_n \to + \infty$.

Firstly, we note that the local maximum points of $u_n$ stay on the compact part $\mathcal{K}$ of $\cG$.
\begin{lemma}\label{lem: stima max}
Let $x_n \in \cG$ be a local maximum point of $u_n$. Then $x_n\in \mathcal{K}$, and
\[
u_n(x_n) \ge \lambda_n^\frac{1}{p-2}.
\]
\end{lemma}
\begin{proof}
Let $x_n$ be in the interior of some half line $l_i\in \mathcal{E}$. In coordinates, $l_i \simeq [0, +\infty)$ and $x_n\in (0, +\infty)$, and the restriction of $u_n$ on $l_i$ is a $C^2$ function on $[0,+\infty)$, by regularity. Since $x_n$ is a local maximum point, we have $u''_n(x_n)\le0$, which implies that $\lambda_n u_n(x_n)\le0$. A contradiction. Thus $x_n\in \mathcal{K}$, and at this point the estimate can be proved as in \cite[Lemma 4.1]{CJS-2022}.
\end{proof}

In the following, we denote by $B_r(x_0)=\{x \in \cG: \ {\rm dist}(x,x_0)<r\}$. Moreover, we denote by $\cG_m$ the star-graph with $m \ge 1$ half-lines glued together at their common origin $0$. In particular, $\cG_1=\R^+$, and $\cG_2$ is isometric to $\R$.

Since $m(u_n)\le2$, we can describe the asymptotic behavior of the sequence $\{u_n\}$ near a local maximum point as $\lambda_n\to +\infty$. From now on, we denote by $\chi_A$ the characteristic function of a set $A$.

\begin{proposition}\label{thm: blow-up 1}
Let $x_n \in \cG$ be such that, for some $R_n \to \infty$,
\begin{equation*}\label{hp max}
u_n(x_n) = \max_{B_{R_n \tilde \eps_n}(x_n)} u_n \quad \text{where } \tilde \eps_n=(u_n(x_n))^{-\frac{p-2}{2}} \to 0.
\end{equation*}
Suppose moreover that
\begin{equation}\label{int max}
\limsup_{n \to \infty} \frac{{\rm{dist}}(x_n, \mathcal{V})}{\tilde \eps_n} = +\infty.
\end{equation}
Then, up to extraction of a subsequence, points ($i$)-($iv$) in \cite[Theorem 4.2]{CJS-2022} holds. 
If, instead of \eqref{int max}, we suppose that
\begin{equation}\label{ver max}
\limsup_{n \to \infty} \frac{{\rm{dist}}(x_n, \mathcal{V})}{\tilde \eps_n} <+\infty;
\end{equation}
then, up to extraction of a subsequence :
\begin{itemize}
\item[($i'$)] $x_n \to {\rm{v}} \in \mathcal{V}$, and all the $x_n$ lie on the same edge ${\rm{e}}_1 \simeq [0,\rho_1]$, where ${\rm{e}}_1\subset \mathcal{K}$ and the vertex ${\rm{v}}$ is identified by the coordinate $0$ on ${\rm{e}}_1$.
\item[($ii'$)] Let ${\rm{e}}_2 \simeq [0,r_2]$, \dots, ${\rm{e}}_j \simeq [0,r_j]$ be the other bounded edges of $\cG$ having ${\rm{v}}$ as a vertex (if any), and let $l_{j+1} \simeq [0,+\infty)$, \dots, $l_m \simeq [0,+\infty)$ be the half-lines of $\cG$ having $\mathrm{v}$ as a vertex (if any), where ${\rm{v}}$ is identified by the coordinate $0$ on each ${\rm{e}}_i$ and $l_i$. Setting $\eps_n = \lambda_n^{-\frac12}$, we have that
\begin{equation}\label{rel eps tilde eps 2}
\begin{split}
\frac{\tilde \eps_n}{\eps_n} &\to (0,1], \\
\limsup_{n \to \infty} &\frac{{\rm{dist}}(x_n, \mathcal{V})}{\eps_n} < +\infty,
\end{split}
\end{equation}
and the scaled sequence defined by
\[
v_n(y):= \eps_n^{\frac{2}{p-2}} u_n(\eps_n y) \quad \text{for }y \in \frac{{\rm{e}}_i}{\eps_n},  \text{ with $i=1,\dots,j$}, \quad \text{or for }y \in \frac{l_i}{\eps_n},  \text{ with $i=j+1,\dots,m$}
\]
converges to a limit $V$ in $C^0_{\rm{loc}}(\cG_m)$ as $n \to \infty$. Denoting by $V_i$ the restriction of $V$ to the $i$-th half-line $\ell_i$ of $\cG_m$, and by $v_{i,n}$ the restriction of $v_n$ to ${\rm e}_i/\eps_n$, we have moreover that $v_{i,n} \to V_i$ in $C^2_{\rm{loc}}([0,+\infty))$. Finally, $V \in H^1(\cG_m)$ is a positive solution to the NLS-type equation on the star-graph
\[
\begin{cases}
-V'' + V = V^{p-1} \chi_{\ell_1 \cup \dots \cup \ell_j}, \quad V>0 & \text{in $\cG_m$},\\
\sum_{i=1}^m V_i'(0^+) =0, \\
V(x) \to 0 & \text{as }\dist(x,0) \to \infty
\end{cases}
\]
with a global maximum point $\bar x$ located on $\ell_1$, whose coordinate is
\[
\bar x = \lim_{n \to \infty} \bar x_n \in [0,+\infty), \quad \text{where} \quad \bar x_n := \frac{{\rm{dist}}(x_n, \mathcal{V})}{\eps_n} .
\]
\item[($iii'$)] There exists $\phi_n \in H^1(\cG_m) \cap C_c(\cG_m)$, with ${\rm{supp}}\,\phi_n \subset B_{\bar R \eps_n}(x_n)$ for some $\bar R>0$, such that
\[
Q(\phi_n; u_n, \cG) <0.
\]
\item[($iv'$)] For all $R>0$ and $q \ge 1$, we have that
\[
\lim_{n \to \infty} \lambda_n^{\frac12- \frac{q}{p-2}} \int_{B_{R \eps_n}(x_n)} u_n^{q}\,dx =   \lim_{n \to \infty} \int_{B_R(\bar x_n)} v_n^q\,dy = \int_{[0, \bar x + R]} V_1^q\,dy + \sum_{i=2}^m \int_{[0, R- \bar x]} V_i^q\,dy  = \int_{B_R(\bar x)} V^q\,dy
\]
(where $B_R(\bar x_n)$ and $B_R(\bar x)$ denote the balls in the scaled and in the limit graphs, respectively).
\end{itemize}
\end{proposition}

The proof of the proposition can be obtained by making minor modifications to the one of Theorem 4.2 in \cite{CJS-2022}. Roughly speaking, with respect to the setting in \cite{CJS-2022}, we have to take into account the fact that the points $x_n$ (which belong to the compact core of $\cG$, by Proposition \ref{lem: stima max})  may accumulate on a vertex ${\rm{v}}$ of $\cG$ which belongs to some half-line. We will need the following two statements, which correspond to Lemmata 4.3 and 4.5 in \cite{CJS-2022}. 

\begin{lemma}\label{lem: sol to 0}
Let $U \in H^1_{{\rm loc}}(\cG_m)$ be a solution to 
\begin{equation}\label{nls on star}
\begin{cases}
-U'' + \lambda U = \rho U^{p-1} \chi_{\ell_1 \cup \dots \cup \ell_j} & \text{in $\cG_m$}, \\
U>0 & \text{in $\cG_m$}, \\
\sum_{i=1}^m U_i'(0) = 0,
\end{cases}
\end{equation}
for some $p>2$, $\rho, \lambda>0$, where $U_i$ denotes the restriction of $U$ on the $i$-th half-line $\ell_i$ of $\cG_m$, and $1<j \le m$. Suppose that $U$ is bounded, and that $U$ is stable outside a compact set $K$, in the sense that $Q(\varphi; U, \cG_m) \ge 0$ for all $\varphi \in H^1(\cG_m) \cap C_c(\cG_m \setminus K)$. Then $U(x) \to 0$ as ${\rm{dist}}(x,0) \to +\infty$, and $U \in H^1(\cG)$.
\end{lemma}

\begin{proof}
On $\ell_1, \dots, \ell_j$, one can argue exactly as in \cite{CJS-2022} (which in turn follows \cite[Theorem 2.3]{EspPet}). On $\ell_{j+1}, \dots, \ell_m$, we note that $U$ satisfies $U'' = \lambda U$ for some $\lambda>0$, and, since $U$ is bounded, we deduce that the restriction of $U$ on each of these half-lines has the form $U(x) = c e^{-\sqrt{\lambda} x}$, for some $c>0$. 
\end{proof}

\begin{lemma}\label{lem: Morse pos}
Let $U \in H^1(\cG_m)$ be any non-trivial solution of \eqref{nls on star}. Then its Morse index $m(U)$ is strictly positive.
\end{lemma}

\begin{proof}
Thanks to the Kirchhoff condition
\[
\int_{\cG_m} \left(|U'|^2 + \lambda U^2\right) dx = \sum_{i=1}^j \int_{\ell_i} \rho |U|^p\, dx.
\]
Therefore
\[
Q(U; U, \cG_m) = (2-p) \sum_{i=1}^j\int_{\ell_i} |U|^p\,dx<0
\]
(recall that $j > 1$), and the thesis follows by density of $H^1(\cG_m) \cap C_c(\cG_m)$ in $H^1(\cG_m)$.
\end{proof}

\begin{proof}[Proof of Proposition \ref{thm: blow-up 1}]
We only emphasize the main differences with respect to \cite{CJS-2022}. Since $\{x_n\} \subset \mathcal{K}$ by Lemma \ref{lem: stima max}, and $\cG$ has a finite number of vertices and edges, up to extraction of a subsequence, the maximum points $x_n$ belong to the same bounded edge ${\rm{e}_1} \simeq [0,r_1]$, and converge. If \eqref{int max} holds, then the proof proceeds exactly as the one of \cite[Theorem 4.2 under assumption (4.3)]{CJS-2022}. If instead \eqref{ver max} holds, then $x_n \to {\rm{v}} \in \mathcal{V}$, since $\tilde \eps_n \to 0$. Contrary to \cite{CJS-2022}, we have to take into account the possibility that ${\rm{v}}$ is on some half-line of $\cG$; thus, the proof does not directly follows from \cite{CJS-2022}, but can be easily adjusted.\\
By the previous discussion, ($i'$) holds, and we can suppose that
\[
\frac{d_n}{\tilde \eps_n} \to \eta \in [0,+\infty), \quad d_n:={\rm dist}(x_n,\mathcal{V}) = x_n.
\]
Let 
\[
\tilde u_{i,n}(y):= \tilde \eps_n^{\frac{2}{p-2}} u_n(\tilde \eps_n y) \quad \text{for }y \in \tilde {\rm e}_{i,n}:= \frac{{\rm e}_i}{\tilde \eps_n}, \ \text{with $i=1,\dots,j$}, \text{ and for }y \in \tilde{l}_{i,n} := \frac{l_i}{\tilde \eps_n}, \ \text{with $i=j+1,\dots,m$},  
\]
and $\tilde u_n =(\tilde u_{1,n}, \dots, \tilde u_{m,n})$. Note that $\tilde u_n$ is defined on a graph $\cG_{m,n}$ consisting in $j$ expanding edges and $m-j$ half-lines, glued together at their common origin, which is identified with the coordinate $0$ on each edge or half-line. In the limit $n \to \infty$, this graph converges to the star-graph $\cG_m$. Clearly, for every $a>\eta +1$ and large enough $n$
\begin{equation}\label{glob max}
\tilde u_{1,n}\left(\frac{d_n}{\tilde \eps_n}\right) = 1 = \max_{B_a(0)} \tilde u_n
\end{equation}
(since $u_n(x_n) = \max_{B_{R_n \tilde \eps_n}(x_n)} u_n$ for some $R_n \to +\infty$),
\[
-\tilde u_n'' + \tilde \eps_n^2 \lambda_n \tilde u_n = \rho_n \tilde u_n^{p-1} \chi_{\tilde{{\rm{e}}}_{1,n} \cup \dots \cup \tilde{{\rm{e}}}_{j,n}}, \quad \tilde u_n >0 
\] 
on any edge of $\cG_{m,n}$, and the Kirchhoff condition at the origin holds. Also, by Lemma \ref{lem: stima max},
\[
\tilde \eps_n^2 \lambda_n \in (0,1] \quad \forall n.
\]
Thus, by elliptic estimates, for $i=1,\dots, j$ we have that $\tilde u_{i,n} \to \tilde u_i$ in $C^2_{\rm loc}([0,+\infty))$ for every $i =1,\dots, j$, and the limit $\tilde u_i$ solves
\begin{equation}\label{limit pb 2}
- \tilde u''_i+\tilde \lambda \tilde u_i = \tilde u_i^{p-1}, \quad \tilde u_i  \ge 0 \quad \text{in $(0,+\infty)$}
\end{equation}
for some $\tilde \lambda \in [0,1]$. Analogously, for $i=j+1,\dots, m$ we have that $\tilde u_{i,n} \to \tilde u_i$ in $C^2_{\rm loc}([0,+\infty))$, and the limit $\tilde u_i$ solves
\begin{equation}\label{limit pb 2-L}
- \tilde u''_i+\tilde \lambda \tilde u_i = 0, \quad \tilde u_i  \ge 0 \quad \text{in $(0,+\infty).$}
\end{equation}
 Moreover, since $\tilde u_{n}$ is continuous on $\cG_{m,n}$ and by uniform convergence, $\tilde u_{i_1}(0) = \tilde u_{i_2}(0)$ for every $i_1 \neq i_2$, so that $\tilde u \simeq (\tilde u_1,\dots, \tilde u_m)$ can be regarded as a function defined on $\cG_m$. Since the convergence $\tilde u_{i,n} \to \tilde u_i$ takes place in $C^2$ up to the origin, the Kirchhoff condition also passes to the limit. Now we exclude the case that $\tilde u \equiv 0$ on some half-line of $\cG_m$. By local uniform convergence, we have that 
\[
\tilde u_1(\eta) = \lim_{n \to \infty} \tilde u_{1,n}\left(\frac{d_n}{\tilde \eps_n}\right) = 1,
\]
and $\eta$ is a global maximum point of $\tilde u$, in view of \eqref{glob max}. This implies that $\tilde u_1>0$ in $(0,+\infty)$, by the uniqueness theorem for ODEs. In turn, the Kirchhoff condition, and the uniqueness theorem for ODEs again, ensure that $\tilde u_i>0$ on $(0,+\infty)$ for every $i$. Finally, we claim that 
\begin{equation}\label{cl: fin morse 2}
\text{the Morse index of $\tilde u$ is at most $2$}.
\end{equation}
The proof of this claim exploits the fact that $m(u_n) \le 2$ for every $n$, is analogue to the one of \cite[Eq. (4.10)]{CJS-2022}, and hence is omitted. 

The case $\tilde \lambda=0$ can be ruled out by phase-plane analysis (the equation $\tilde u_1'' + \tilde u_1^{p-1}=0$ on $\ell_1 = [0,+\infty)$ has only periodic sign-changing solutions, save the trivial one). Therefore, 
\begin{equation}\label{rel max lam}
0<\liminf_{n \to \infty} \frac{\lambda_n}{(u_n(x_n))^{p-2}} \le \limsup_{n \to \infty} \frac{\lambda_n}{(u_n(x_n))^{p-2}} \le 1,
\end{equation}
which proves the first estimate in \eqref{rel eps tilde eps 2}. Furthermore, $\tilde u$ is bounded and hence, by Lemma \ref{lem: sol to 0}, $\tilde u \to 0$ as $|x| \to +\infty$, and $\tilde u \in H^1(\cG_m)$.

At this point it is equivalent, but more convenient, to work with $v_n$ defined in point ($ii'$) of the proposition, rather than with $\tilde u_n$. By \eqref{ver max} and \eqref{rel max lam},
\[
\limsup_{n \to \infty} \frac{{\rm dist}(x_n,\mathcal{V})}{\eps_n} <+\infty.
\]
Thus, similarly to what was done before, one can show that $v_n$ converges, in $C^0_{\rm loc}(\cG_m)$ and in $C^2_{{\rm loc}}([0,+\infty))$ on every half-line, to a limit function $V\simeq(V_1,\dots,V_m)$, which solves 
\begin{equation}\label{limit pb v 2}
\begin{cases}
- V''+V = V^{p-1} \chi_{\ell_1 \cup \dots \cup \ell_j}, \quad V  \ge 0 & \text{in $\cG_m$}, \\
\sum_{i=1}^m V_i'(0^+) = 0;
\end{cases}
\end{equation}
furthermore, $V$ has a positive global maximum on the half-line $\ell_1$, $V_1(\bar x) \ge 1$ (thus $V>0$ in $\cG_m$), and has finite Morse index $m(V) \le 2$. Moreover, by Lemma \ref{lem: sol to 0}, $V \to 0$ as $|x| \to \infty$. Thus, ($ii'$) is proved. Point ($iv'$) follows directly by local uniform convergence. Finally, point ($iii'$) is a consequence of Lemma \ref{lem: Morse pos}. This implies that there exists $\phi \in H^1(\cG_m) \cap C_c(\cG_m)$ such that $Q(\phi; V,\cG_m)<0$; thus, defining
\[
\phi_{i,n}(x)= \eps_n^{\frac{1}{2}} \phi_i\left(\frac{x-x_n}{\eps_n}\right),
\]
it is not difficult to deduce that for sufficiently large $n$ we have $Q(\phi_{i,n};u_n, \cG)<0$, and ${\rm supp}\, \phi_{i,n} \subset B_{\bar R \eps_n}(x_n)$ for some positive $\bar R$.
\end{proof}

In what follows, we focus on the global behavior of $\{u_n\}$ and show the way that $\{u_n\}$ decays exponentially away from local maximum points. In fact, by similar arguments to those used to establish Theorem 4.6 in \cite{CJS-2022} (see also \cite{DGMP,Espetal,EspPet,PieVer}), we have the following result.

\begin{proposition}\label{thm: blow-up 2}
There exist $k \in \{1,2\}$, and sequences of points $\{P_n^i\}$, $i \in \{1,k\}$, such that
\begin{gather}
\lambda_n \dist(P_n^1, P_n^2) \to +\infty \quad ~(\mbox{if}~k=2),\label{36ep} \\
u_n(P_n^i) = \max_{B_{R_n \lambda_n^{-1/2}}(P_n^i)} u_n \quad \text{for some $R_n \to +\infty$, for every $i$,} \label{37ep}
\end{gather}
and constants $C_1, C_2>0$ such that
\begin{equation*}\label{estim far away}
u_n(x) \le C_1 \lambda_n^{\frac{1}{p-2}} \sum_{i=1}^k e^{-C_2 \lambda_n^{\frac{1}{2}} \dist(x, P_n^i)}+ C_1 \lambda_n^{\frac{1}{p-2}}\sum_{j=1}^{m_3} e^{-C_2 \lambda_n^{\frac{1}{2}} \dist(x, {\rm v}_j )},   \quad \forall x \in \cG \setminus \bigcup_{i=1}^k B_{R \lambda_n^{-1/2}} (P_n^i),
\end{equation*}
where ${\rm v}_1, \dots, {\rm v}_{m_3}$ are all the vertices of $\cG$.
\end{proposition}

\begin{proof}
Having proved Proposition \ref{thm: blow-up 1}, we can adjust the proof of \cite[Theorem 4.6]{CJS-2022} with minor changes.

\emph{Step 1)  There exist $k \in \{1,2\}$, and sequences of points $\{P_n^i\}$, $i \in \{1,k\}$, such that \eqref{36ep} and \eqref{37ep} hold, and moreover
\begin{equation}\label{39ep}
\lim_{R \to +\infty} \left( \limsup_{n \to \infty} \  \lambda_n^{-\frac{1}{p-2}} \max_{d_n(x) \ge R \lambda_n^{-1/2}} u_n(x) \right) = 0.
\end{equation}
where $d_n(x) =\min\{ \dist(x,P_n^i): \ i=1,k\}$ is the distance function from $\{P_1^n, P_k^n\}$.}

The proof of this point, as in \cite{CJS-2022}, can be directly adapted from \cite[Theorem 3.2]{EspPet}.

\medskip

\emph{Step 2) Conclusion of the proof.} Here we argue as in the second step of the proof of \cite[Theorem 4.6]{CJS-2022}. Therein, a comparison argument is used based on the following estimates: on one side, for all small $\eps \in (0,1)$, there exist $R>0$ and $n_R \in \mathbb{N}$ such that 
\begin{equation}\label{4.20 art1}
\max_{d_n(x) > R \lambda_n^{-1/2}} u_n(x) \le \lambda_n^{\frac{1}{p-2}} \eps, \quad \forall n \ge n_R.
\end{equation}
Moreover, in $A_n:= \{d_n(x) > R \lambda_n^{-1/2}\}$ it results that
\begin{equation}\label{4.21 art1}
-u_n''+\frac{\lambda_n}{2} u_n \le 0.
\end{equation}
Once we have checked that these estimates are satisfied also in the present setting, we can follow \cite{CJS-2022} verbatim. 

Inequality \eqref{4.20 art1} is given by the first step of this proof. Regarding \eqref{4.21 art1}, its validity on $\mathcal{K} \cap A_n$ follows from \eqref{4.20 art1}, since on this set
\[
u_n'' = (\lambda_n-u_n^{p-2}) u_n 
\]
and $\lambda_n-u_n^{p-2} \ge \lambda_n/2$ provided that $\eps >0$ is small enough. On $\cG \setminus \mathcal{K}$, instead, we can directly observe that 
\[
u_n'' = \lambda_n u_n \ge \frac{\lambda_n}2 u_n
\]
since $\lambda_n >0$ and $u_n \ge 0$. Therefore, \eqref{4.20 art1} and \eqref{4.21 art1} holds, and the thesis follows.
\end{proof}

Now, we can give

\begin{proof}[Proof of Theorem \ref{thm: main ex}]
We have already observed that if  $\{\lambda_n\}$ is  bounded then also $\{u_{\rho_n}\}$ is bounded. We shall thus assume by contradiction that  $\lambda_n \to +\infty$, up to extraction of a subsequence. Then, Propositions \ref{thm: blow-up 1} and \ref{thm: blow-up 2} hold for $u_n:=u_{\rho_n}$ and by similar arguments as in the proof of Proposition 5.1 in \cite{CJS-2022}, we can get a contradiction. \smallskip

 Since $\{\lambda_n\}$ is bounded, passing to a subsequence, there exists $\lambda^*\in \mathbb{R}$ such that $\lim\limits_{n\to+\infty}\lambda_n=\lambda^*$. Also, since $\{u_n\}$ is  bounded  and $u_n>0$ on $\cG$, there exists $u^*\in H^1(\mathcal{G})$ such that
 \begin{eqnarray*}
 &&u_{n}\rightharpoonup u^*~~\mbox{in}~~ H^1(\mathcal{G}),\\
 &&u_n\to u^*~~\mbox{in}~~L_{{\rm{loc}}}^r(\mathcal{G}), r>2,\\
 &&u_{n}(x)\to u^*(x)~~\mbox{for a.e.}~ x\in \mathcal{G},
 \end{eqnarray*}
which implies that $u^*\ge0$ and $u^*$ satisfies
\begin{equation*}
- (u^*)'' + \lambda^* u^*=\kappa(x) (u^*)^{p-1}
\end{equation*}
together with the Kirchhoff condition
\begin{eqnarray*}
\sum_{{\rm e}_i\succ {\rm v}_j}(u^*)'({\rm v}_j)+\sum_{l_i\succ {\rm v}_j}(u^*)'({\rm v}_j)=0, \qquad \forall j=1,\cdots, m_3.
\end{eqnarray*}
By similar arguments to those in the proof of Proposition \ref{prop: ex for ae rho}, we can deduce that $\lambda^* >0$, $u_n \to u^*$ strongly in $H^1(\mathcal{G})$, and that $u^*>0$. This completes the proof.
\end{proof}

\end{document}